\documentclass[11pt,leqno]{amsart}
\usepackage{amsthm,amsfonts,amssymb,amsmath,oldgerm}
\numberwithin{equation}{section}
\usepackage{fullpage}

\renewcommand\d{\partial}

\def\eps{\varepsilon }

\renewcommand\d{\partial}

\newcommand\R{\mathbb R}

\def\eps{\varepsilon}

\newcommand\br{\begin{remark}}
\newcommand\er{\end{remark}}
\newcommand\bp{\begin{pmatrix}}
\newcommand\ep{\end{pmatrix}}
\newcommand{\be}{\begin{equation}}
\newcommand{\ee}{\end{equation}}
\newcommand\ba{\begin{equation}\begin{aligned}}
\newcommand\ea{\end{aligned}\end{equation}}

\newcommand{\bap}{\begin{app}}
\newcommand{\eap}{\end{app}}
\newcommand{\begs}{\begin{exams}}
\newcommand{\eegs}{\end{exams}}
\newcommand{\beg}{\begin{example}}
\newcommand{\eeg}{\end{exaplem}}
\newcommand{\bpr}{\begin{proposition}}
\newcommand{\epr}{\end{proposition}}
\newcommand{\bt}{\begin{theorem}}
\newcommand{\et}{\end{theorem}}
\newcommand{\bc}{\begin{corollary}}
\newcommand{\ec}{\end{corollary}}
\newcommand{\bl}{\begin{lemma}}
\newcommand{\el}{\end{lemma}}
\newcommand{\bd}{\begin{definition}}
\newcommand{\ed}{\end{definition}}
\newcommand{\brs}{\begin{remarks}}
\newcommand{\ers}{\end{remarks}}

\newcommand{\N}{\mathcal{N}}

\newcommand{\RR}{{\mathbb R}}

\newcommand{\Id}{{\rm Id }}
\newcommand{\Range}{{\rm Range }}

\newtheorem{theorem}{Theorem}[section]
\newtheorem{proposition}[theorem]{Proposition}
\newtheorem{corollary}[theorem]{Corollary}
\newtheorem{lemma}[theorem]{Lemma}

\theoremstyle{remark}
\newtheorem{remark}[theorem]{Remark}
\theoremstyle{definition}
\newtheorem{definition}[theorem]{Definition}

\newtheorem{example}[theorem]{Example}

\newcommand\cT{{\mathcal T}}
\newcommand\cS{{\mathcal S}}
\newcommand\cR{{\mathcal R}}
\newcommand\cQ{{\mathcal Q}}

\newcommand{\RM}{\mathbb{R}}
\newcommand{\ZM}{\mathbb{Z}}

\newcommand{\beq}{\begin{equation}}
\newcommand{\eeq}{\end{equation}}

\newcommand{\ks}{k_\star}

\title{
Nonlocalized modulation of periodic reaction diffusion waves:
Nonlinear stability
}

\author{ Mathew A. Johnson}
\address{ Indiana University, Bloomington, IN 47405}
\email{matjohn@indiana.edu}
\thanks{ Research of M.J. was partially supported by an NSF Postdoctoral Fellowship under NSF grant DMS-0902192}
\author{Pascal Noble}
\address{Universit\'e Lyon I, Villeurbanne, France}
\email{noble@math.univ-lyon1.fr}
\thanks{Research of P.N. was partially supported by the French ANR Project no.
ANR-09-JCJC-0103-01}
\author{L.Miguel Rodrigues}
\address{Universit\'e Lyon 1, Villeurbanne, France}
\email{ rodrigues@math.univ-lyon1.fr}
\thanks{ Stay of M.R. in Bloomington was supported by 
French ANR Project no. ANR-09-JCJC-0103-01}
\author{Kevin Zumbrun}
\address{Indiana University, Bloomington, IN 47405}
\email{kzumbrun@indiana.edu}
\thanks{Research of K.Z. was partially supported
under NSF grant no. DMS-0300487}
\begin{document}

\begin{abstract}
By a refinement of the technique used by Johnson and Zumbrun
to show stability under localized perturbations,
we show that spectral stability implies nonlinear modulational
stability of periodic traveling-wave solutions of reaction
diffusion systems under small perturbations consisting of a
nonlocalized modulation plus a localized perturbation.
The main new ingredient is a detailed analysis of linear behavior
under modulational data $\bar u'(x)h_0(x)$, where $\bar u$ is
the background profile and $h_0$ is the initial modulation.
\end{abstract}
\date{\today}
\maketitle

\section{Introduction}\label{s:introduction}
Stability and behavior of modulated periodic wave trains
have received considerable recent attention, both in the reaction diffusion
and conservation law contexts; see, for example
\cite{S1,S2,DSSS,JZ1,JZ2,JZN,BJNRZ1,BJNRZ2,NR1,NR2}
and references therein.
The initial mathematical challenge of this problem is that the linearized
equations, being periodic-coefficient, have purely essential spectrum
when considered as problems on the whole line,
so that there is no spectral gap between neutral and other modes,
making difficult either the treatment of linearized behavior or
the passage from linear to nonlinear estimates.

This issue was overcome in the reaction diffusion context in the late 1990s by Schneider \cite{S1,S2},
resolving an at the time 30-year open problem.
Using a method of ``diffusive stability," Schneider combined diffusive-type linear estimates
with renormalization techniques to show that, assuming ``diffusive" spectral stability in the sense to be specified
later of the background periodic wave, long-time behavior under localized perturbations is essentially described
by a scalar heat equation in the phase variable, with the amplitude decaying more rapidly.

More recently, Oh--Zumbrun and Johnson--Zumbrun \cite{OZ1,JZ2} in 2010
using rather different tools coming from viscous shock theory
have resolved the corresponding problem in the conservation
law case, showing that, again assuming diffusive spectral stability of
the background periodic wave, behavior under 
localized perturbations
is described, roughly, by a system of viscous conservation laws in
the derivative of the phase and other modulation parameters,
so that the phase decays at slower, errorfunction rate than
in the reaction diffusion case.
This description
has since been sharpened in \cite{JNRZ2},
showing that behavior consists of a modulation governed by the
associated
formal second-order Whitham averaged system plus a faster-decaying term.
Further developments
include the treatment of nonlinear stability of roll waves
of the St. Venant equations \cite{JZN,BJNRZ1}
and the resolution in \cite{BJNRZ2}
of the 35-year open problem of proving
nonlinear stability of
spectrally stable periodic Kuramoto--Sivashinsky waves.

Applied to the reaction diffusion context \cite{JZ1},
the approach of \cite{JZ2}
recovers and slightly extends the results of Schneider, yielding
the same heat kernel rate of decay
with respect to localized perturbations while
removing the assumption of \cite{S1,S2} that nearby periodic
waves have constant speed.
However, one might hope to do better, based on the successful
analysis in \cite{JZ2}, etc., in the conservation law case, of perturbations
decaying at slower, errorfunction rate.

As described
clearly
in \cite{DSSS},
formal (WKB) asymptotics suggest the same thought, yielding
the description to lowest order of a scalar viscous conservation law
\be\label{whit}
k_t-\omega(k)_x= (d(k)k_x)_x
\ee
for the wave number $k=\psi_x$, where frequency 
$\omega=\psi_t$
and $\psi$ is the phase. Here,
$\omega(k)=-c(k)k$
and wave speed $c(k)$ are determined through
the nonlinear dispersion
relation obtained from the manifold of periodic traveling-wave solutions
$\bar u^k(k(x-ct))$,
and $d(k)$ through higher-order asymptotics.
See \cite{Se,OZ2,NR1,NR2} for corresponding developments
in the conservation law case.
From \eqref{whit} and the definition $k=\psi_x$, one might hope
to treat initial modulations $\psi{|}_{t=0}=h_0$
for which, not the phase $h_0$, but its derivative $\partial_x h_0$ is
localized ($L^1$), obtaining the same heat kernel rate of decay for $k$
as found in the conservation law case \cite{JZ2}.

In this note, we show that this is indeed the case,
proving by an  adaptation of the methods of \cite{JZ1,JZ2} that
{\it diffusively spectrally stable periodic reaction diffusion waves
are nonlinearly stable under initial perturbations consisting of a localized
perturbation plus a nonlocalized modulation $h_0$, $\partial_x h_0\in L^1$,
in phase}, with heat kernel rate of decay in the wave number and
errorfunction decay in the phase.
In a companion paper \cite{JNRZ1},
we build on these basic estimates to establish, further,
time-asymptotic behavior,
validating description \eqref{whit} in the long-time limit.
Recall that \eqref{whit} has been validated in \cite{DSSS} in the  related,
asymptotically-large-bounded-time small-wavenumber limit and  for data in uniformly-local
Sobolev spaces\footnote{That is, following the notation of \cite{DSSS}, $k=\psi_x\in  H^s_{ul}$ with small norm $\|w\|_{H^s_{ul}}:=\sup_{x\in\RM}\|w\|_{H^s([x,x+1])}$ and $s\geq M+2\geq 4$.}, in the somewhat different sense of showing
that there  exists a $\delta$-family of solutions of the full system  $\delta^M$-close to
a formal expansion\footnote{But this expansion is only $\delta^2$-close to the approximate solution involving only the solution of the second order Whitham equation.} in $\delta$ on intervals
$[0,T/\delta^2]$, where $M$ and $T$ are
arbitrarily large and $\delta$ is the  characteristic size of the wavenumber of the modulation:
that is,  they build ``prepared'' data for bounded intervals rather than work with  general
ones globally in time.

\medskip
Consider a periodic traveling-wave solution of reaction diffusion equation
$u_t=u_{xx} + f(u)$, or, equivalently, a standing-wave solution
$u(x,t)=\bar u(x)$ of
\be\label{rd}
\ks {u}_t=\ks^2{u}_{xx}+f({u})+\ks c{u}_x,
\ee
where $c$ is the speed of the original traveling wave,
and wave number $k_*$ is chosen so that
\be\label{per}
\bar u(x+1)=\bar u(x).
\ee
For simplicity of notation, we will follow this convention throughout
the paper; that is, all periodic functions are assumed to be
periodic of period one.

We make the following standard genericity assumptions:
\begin{enumerate}
  \item[(H1)] $f\in C^K(\RM)$,
($K\ge3$).
  \item[(H2)]  Up to translation, the set of $1$-periodic solutions of \eqref{rd} (with $k$ replacing $\ks$) in the vicinity of $\bar{u}$, $k=\ks$,
forms a smooth $1$-dimensional manifold
$\{\bar{u}(k,\cdot)\}=\{\bar{u}^k(\cdot)\}$, $c=c(k)$.
\end{enumerate}

Linearizing \eqref{rd} about $\bar u$ yields the periodic coefficient
equation
\be\label{lin}
\ks v_t=\ks Lv:= (\ks^2\partial_x^2 +\ks c\partial_x +b)v,
\qquad
b(x):=df(\bar u(x)).
\ee
Introducting the family of Floquet operators
\be\label{Lxi}
\ks L_\xi:= e^{-i\xi x}\ks L e^{i\xi x}=
\ks^2(\partial_x +i\xi)^2 + \ks c(\partial_x+i\xi) + b,
\ee
operating on periodic functions on $[0,1]$,
determined by the defining relation
\be\label{defrel}
L (e^{i\xi x}f)= e^{i\xi x} (L_\xi f)
\; \hbox{\rm  for $f$ periodic},
\ee
we define following \cite{S1,S2}
the standard {\it diffusive spectral stability} conditions:
\begin{enumerate}
  \item[(D1)]
$\lambda=0$ is a simple eigenvalue of $L_0$.
(Note that $\xi=0$ corresponds to co-periodic perturbations,
hence $L_0$ has always at least the translational, zero-eigenfunction $\bar u'$.)
  \item[(D2)]
$\sigma(L_{\xi})\subset\{\lambda\ |\ \Re \lambda\leq-\theta|\xi|^2\}$
for some constant $\theta>0$.
\end{enumerate}

Then, we have the following Main Theorem, extending
the results of \cite{JZ1} to nonlocalized perturbations.

\begin{theorem}\label{main}
Let $K\ge 3$.
Assuming (H1)-(H2) and (D1)-(D2),
let
$$
E_0:=\|\tilde u_0(\cdot-h_0(\cdot))-\bar u(\cdot)\|_{L^1(\RM)\cap H^K(\RM)}
+\|\partial_x h_0\|_{L^1(\RM)\cap H^K(\RM)}
$$
be sufficiently small, for some choice of modulation $h_0$.
Then, there exists a global solution $\tilde u(x,t)$ of \eqref{rd}
with initial data $\tilde u_0$ and a phase function
$\psi(x,t)$ such that, for $t>0$ and $2\le p \le \infty$,
\ba\label{mainest}
\|\tilde u(\cdot-\psi(\cdot,t), t)-\bar u(\cdot)\|_{L^p(\RM)}
&\lesssim E_0 (1+t)^{-\frac{1}{2}(1-1/p)}\\
\|\nabla_{x,t} \psi(\cdot,t) \|_{L^p(\RM)}
&\lesssim E_0 (1+t)^{-\frac{1}{2}(1-1/p)},\\
\ea
and
\ba\label{andpsi}
\|\tilde u(\cdot , t)-\bar u(\cdot)\|_{L^\infty(\RM)}, \quad
\| \psi(\cdot,t) \|_{L^\infty(\RM)} &\lesssim E_0.
\ea
In particular, $\bar u$ is nonlinearly
(boundedly) stable in $L^\infty(\R)$
with respect to initial perturbations
$v_0=\tilde u_0-\bar u$
for which $\|v_0\|_{E}:=\inf_{\partial_x h_0\in L^1(\RM) \cap H^K(\RM)} E_0$ is sufficiently
small.
\end{theorem}

\br\label{variables}
\textup{
It may seem more natural, and indeed is so, to introduce $h_0$ and $\psi$ via
$$
v(x,0)=\tilde u_0(x)-\bar u(x+h_0(x)),\quad v(x,t)=\tilde u(x,t)-\bar u(x+\psi(x,t)).
$$
However, in doing so one introduces in the equation for $v$ terms involving only $\psi$ and thus not decaying in time; see Lemma \ref{lem:canest} below.
For this reason we work instead with
$$
v(x,t)=\tilde u(x-\psi(x,t),t)-\bar u(x),
$$
that is,
$$
\tilde u(x,t)=\bar u(Y(x,t))+v(Y(x,t),t)
$$
where $Y$ is such that
$$
Y(x,t)-\psi(Y(x,t),t)=x,\quad Y(y-\psi(y,t),t)=y \ .
$$
Notice that we insure the existence of such a map $Y$ by keeping, for any $t$, $\|\psi(\cdot,t)\|_{L^\infty(\RM)}$ bounded and $\|\psi_x(\cdot,t)\|_{L^\infty(\RM)}$ small.
It should be stressed, however,
that\footnote{This follows from $Y(X(y,t),t)-[X(y,t)+\psi(X(y,t),t)]=\psi(y,t)-\psi(y-\psi(y,t),t)$.}
$$
Y(x,t)=x+\psi(x,t)+O(\|\psi(\cdot,t)\|_{L^\infty}\|\psi_x(\cdot,t)\|_{L^\infty})
$$
so that we are not so far
from the natural (but inappropriate) approach.  Furthermore, notice also that introducing the map $Y$ above enables one to go back
to the original unknown $\tilde{u}(x,t)$ when desired.
}
\er

At a philosophical level, the main new observation here
beyond what was shown in \cite{JZ1} is that linear perturbations that 
are ``asymptotically modulational''\footnote{In the sense that they converge at plus and minus spatial infinity to 
scalar multiples $h_0 \bar u'$ of the derivative $\bar u'$ of the background traveling wave, corresponding at linear level to shifts $\bar u(x+h_0(x))$.}
 behave as regards the linear solution operator $S(t)$ as if they
were asymptotically constant; in particular, $\partial_x S(t)(\bar u'h_0)\sim S(t)(\bar u' \partial_x h_0)$,
where $\bar u$ is the background traveling wave 
(see Section \ref{s:mod}).
The analysis is based, at a technical level,
on the decomposition of Floquet modes of $\bar u' h_0$
into the sum of Fourier modes of $h_0$ times the periodic function $\bar u'(x)$,
and the separate analysis/recombination of each frequency range
$[-\pi+2j\pi, \pi+2j\pi]$.
We are unable to see such detail from the
Green function description of \cite{JZ1,JZ2,JZN}, nor from weighted
energy estimates as in \cite{S1,S2}.
Indeed, this seems to be an instance where frequency domain techniques
detect cancellation not readily apparent by spatial-domain techniques.

In the process, we obtain an alternative proof of the basic estimates
in \cite{JZ1}, carried out entirely within the Bloch transform formulation
without passing to a Green function description as in \cite{JZ1}.
This elucidates somewhat the relation between the Bloch tranform
based estimates of \cite{S1,S2} and the Green function based
estimates of \cite{JZ1,JZ2}, in particular the role
of the Green function (integral kernel)
in obtaining the estimates of \cite{JZ1}; see Remark \ref{BG}.

Our analysis suggests that one might by an elaboration of the
approach used here treat still more general perturbations
converging asymptotically to {\it any} fixed periodic perturbations
at $\pm \infty$, not necessarily modulations, which would
then time-exponentially relax to modulations following the
dynamics of the corresponding periodic problem;
see Remark \ref{remarks}.2.
This would be an interesting direction for further investigation.

We mention, finally, that the techniques introduced here
are not limited to the reaction-diffusion case, but apply equally
to the conservative case treated in \cite{JZ2}, yielding a comparable
result of stability with respect to nonlocalized 
modulations \cite{JNRZ3}.
Indeed, this clarifies somewhat the relation between the reaction diffusion
and conservation law case, revealing a continuum of models with
common behavior lying between these two extremes.
These issues will be reported on elsewhere.

\br\label{c0}
\textup{
Without change, we may treat reaction-diffusion systems with a 
diagonal diffusion as considered in \cite{DSSS}.
More generally, all
proofs go through whenever the linearized operator generates
a $C^0$ semigroup;
in particular, our results apply to the
(sectorial) Swift--Hohenberg equations treated for localized
perturbations in \cite{S2}.\footnote{ See \cite{BJNRZ1,BJNRZ2,BJNRZ3,JZN} for related analyses in such general settings.}
It would be interesting to extend to the multi-dimensional
case as in \cite{U}.
}
\er

\medskip

{\bf Note:}
We have been informed by B\"jorn Sandstede that
similar results have been obtained by 
different means
by him and collaborators \cite{SSSU}
using a nonlinear decomposition of phase and amplitude variables as
in \cite{DSSS}.

\section{Preliminaries}\label{s:prelim}

Recall the Bloch solution formula for periodic-coefficient
operators,
\be\label{fullS}
(S(t) g)(x):=(e^{tL}g)(x) = \int_{-\pi}^{\pi} e^{i\xi x} (e^{tL_\xi}  \check g(\xi, \cdot))(x) d\xi,\footnote{
In other words,
$\check{(e^{tL}g)}(\xi,x)= (e^{tL_\xi} \check g(\xi,\cdot))(x)$,
 a consequence of \eqref{defrel}.}
\ee
where
$L_\xi$ is as in \eqref{Lxi},
\be\label{checkg}
\check g(\xi,x):= \sum_{j\in \ZM} \hat g(\xi+2j\pi) e^{i2\pi jx},
\ee
periodic, denotes the Bloch transform of $g$,
$\hat g(\xi):=\frac{1}{2\pi} \int_\RM e^{-i\xi x} g(x) dx$
the Fourier transform, and
\be\label{Brep}
g(x)= \int_{-\pi}^{\pi} e^{i\xi x}  \check g(\xi,x) d\xi
\ee
the inverse Bloch transform, or Bloch representation of $g\in L^2$.

Note that, in view of the inverse Bloch transform formula \eqref{Brep}, the generalized Hausdorff--Young inequality
$\|u\|_{L^p(\RR)} \le \| \check u\|_{L^q(\xi,L^p([0,1]))}$
for $q\le 2\le p $ and $\frac{1}{p}+\frac{1}{q}=1$ \cite{JZ1}, yields for any $1$-periodic
functions $f(\xi,\cdot)$, $\xi\in[-\pi,\pi]$,
\be\label{hy}
\|\int_{-\pi}^{\pi}e^{i\xi \cdot} f(\xi,\cdot)d\xi\|_{L^p(\RR)}
\le \|f\|_{L^q(\xi,L^p([0,1]))}
\:\; {\rm for} \:\;
q\le 2\le p \:\; {\rm and } \:\; \frac{1}{p}+\frac{1}{q}=1.\footnote{
Here, and elsewhere, we are adopting the notation $\|f\|_{L^q(\xi,L^p([0,1]))}:=\left(\int_{-\pi}^\pi\|f(\xi,\cdot)\|_{L^p([0,1])}^{q}d\xi\right)^{1/q}$.}
\ee
This convenient formulation is the one by which we will obtain
all of our linear estimates.

Note by (D1) that there exists a simple zero eigenfunction $\bar u'$ of $L_0$,
which by standard perturbation results \cite{K}
thus bifurcates to an eigenfunction $\phi(\xi,\cdot)$, with associated
left eigenfunction $\tilde \phi(\xi, \cdot)$ and eigenvalue
\be\label{lambda}
\lambda(\xi)= ai \xi - d \xi^2 + O(|\xi|^3),
\ee
where $a$ and $d$ are real and $ d>0$ by assumption (D2) and
complex symmetry, $\lambda(\xi)=\bar \lambda(-\xi)$,
each of $\phi$, $\tilde \phi$, $\lambda$ analytic in $\xi$ and defined
for $|\xi|$ sufficiently small.

\section{Basic linear estimates}\label{s:basic}

Loosely following \cite{JZ1} decompose the solution operator as
\be\label{decomp}
S(t)=S^p(t)+ \tilde S(t),
\qquad
S^p(t)=\bar u'\ s^p(t),
\ee
with
\ba\label{sp}
(s^p(t)g)(x)&=\int_{-\pi}^{\pi}
e^{i\xi x}\alpha(\xi) e^{\lambda(\xi)t} \langle \tilde \phi(\xi,\cdot), \check g(\xi,\cdot)\rangle_{L^2([0,1])} d\xi,
\ea
and
\ba\label{tildeS}
(\tilde S(t) g)(x)&:=
\int_{-\pi}^{\pi} e^{i\xi x} (1-\alpha(\xi))
(e^{L_\xi t}  \check g(\xi))(x) d\xi
+ \int_{-\pi}^{\pi} e^{i\xi x} \alpha(\xi)
(e^{L_\xi t} \tilde \Pi(\xi) \check g(\xi))(x) d\xi\\
&
+ \int_{-\pi}^{\pi}
e^{i\xi x}\alpha(\xi) e^{\lambda(\xi) t}(\phi(\xi,x)-\phi(0,x))
\langle \tilde \phi(\xi),\check g(\xi) \rangle_{L^2([0,1])} d\xi
,
\ea
where $\alpha$ is a smooth cutoff function supported on $\xi$ sufficiently
small,\footnote{That is, $\alpha=0$ for
$|\xi|\ge 2\xi_0$ and $\alpha=1$ for $|\xi|\le \xi_0$,
$\xi_0$ sufficiently small.}
 \be\label{Pi}
\Pi^p(\xi):= \phi(\xi)\langle \tilde \phi(\xi), \cdot\rangle_{L^2([0,1])}
\ee
defined for sufficiently small $\xi$
denotes the eigenprojection onto the eigenspace
$\Range \{\phi(\xi)\}$ bifurcating from
$\Range\{\bar u'(x)\}$ at $\xi=0$, $\tilde \phi$ the associated left
eigenfunction, and $\tilde \Pi:=\Id-\Pi^p$.

We begin by reproving the following estimates established 
(in a slightly different form) in \cite{JZ1},
describing linear behavior under a localized perturbation $g\in L^1(\RM)\cap L^2(\RM)$.

\bpr[\cite{JZ1}]\label{greenbds}
Under assumptions (H1)-(H2) and (D1)-(D2),
for all $t>0$, $2\leq p\leq \infty$,
\begin{align}\label{finale}
\left\|
\partial_x^l\partial_t^m s^p(t) \partial_x^n g \right\|_{L^p(\RM)}
\lesssim
\min \begin{cases}
(1+t)^{-\frac{1}{2}(1-1/p)-\frac{l+m}{2}}\|g\|_{L^1(\RM)},\\
(1+t)^{-\frac{1}{2}(1/2-1/p)-\frac{l+m}{2}}
\|g\|_{L^2(\RM)},
\end{cases}
\end{align}
for
$0\leq n
\leq K+1$,
and for some $\eta>0$
and $0\leq l+2m,n \leq K+1$,
\begin{align}\label{finalg}
\left\|\partial_x^l \partial_t^m \tilde S(t) \partial_x^n g
 \right\|_{L^p(\RM)}
&\lesssim
\min
\begin{cases}
(1+t)^{-\frac{1}{2}(1-1/p)-\frac{1}{2}}\| g\|_{L^1(\RM)\cap H^{l+2m+1}(\RM)},\\
e^{-\eta t}\|\partial_x^n g\|_{H^{l+2m+1}(\RM)}+(1+t)^{-\frac{1}{2}(1/2-1/p)-\frac{1}{2}}\|g\|_{L^2(\RM)},\\
\end{cases}
\end{align}
\epr

Estimates \eqref{finale}--\eqref{finalg} were established
in \cite{JZ1} by first passing to a Green function formulation.
Here, we give an alternative proof within the Bloch formulation
that will be useful for what follows.

\begin{proof}
{\it (i) (Proof of \eqref{finale}(1)).}
First, expand
\ba\label{0Shf}
(s^p(t) g)(x)&=
\int_{-\pi}^{\pi} \alpha(\xi) e^{\lambda(\xi)t}e^{i\xi x}
\langle \tilde \phi  , \check g\rangle_{L^2([0,1])}(\xi) d\xi\\
&=
\sum_{j\in \ZM}
\int_{-\pi}^{\pi}\alpha(\xi) e^{\lambda(\xi)t}e^{i\xi x}
\langle \tilde \phi(\xi,y),e^{i 2\pi j y}
\rangle_{L^2([0,1])}\hat g(\xi+2j\pi) d\xi\\
&=
\sum_{j\in \ZM}
\int_{-\pi}^{\pi} \alpha(\xi) e^{\lambda(\xi)t}e^{i\xi x}
\hat{\tilde \phi  }_{j}(\xi)^*
\hat g(\xi+2j\pi) d\xi,
\ea
where
$\hat{\tilde \phi }_{j}(\xi)$ denotes the $j$th Fourier coefficient in the
Fourier expansion of periodic function $\tilde \phi(\xi,\cdot)$, and $z^*=\bar z$ denotes complex conjugate.

By Hausdorff-Young's inequality,
$|\hat g|\le \|g\|_{L^1(\RM)}$, and by (D2), $|e^{\lambda(\xi)t}\alpha^{1/2}(\xi)|
\le e^{-\eta \xi^2 t}$, $\eta>0$, while by
by Cauchy--Schwarz' inequality,
\[
\alpha^{1/2}(\xi)\sum_j |\hat{\tilde  \phi}_j(\xi)|\le
\alpha^{1/2}(\xi)\sqrt{\sum_j (1+|j|^2)|\hat {\tilde \phi}_j(\xi)|^2
\sum_j (1+|j|^{-2})}
\le
C\alpha^{1/2}(\xi) \|\tilde \phi(\xi)\|_{H^1(\RM;dx)}.
\]
Combining these facts, and applying
\eqref{hy},
we obtain for $1/q+1/p=1$
$$
\|s^p(t)g\|_{L^p(\RM;dx)}\lesssim
\|e^{-\eta \xi^2 t}\|_{L^q([-\pi,\pi],d\xi)}
\sup_{|\xi|\le \eps}\|\tilde \phi\|_{H^1([0,1];dx)}\|g\|_{L^1(\RR)}
\lesssim
(1+t)^{-\frac{1}{2}(1-1/p)}\|g\|_{L^1(\RR)},
$$
yielding the result for $l=m=n=0$.
Estimates for general $l,m,n \ge 0$ go similarly, passing $\partial_x^n$
derivatives onto $\tilde \phi$ in the inner product using integration
by parts and noting that $\partial_x^l$ and $\partial_t^m$ derivatives
bring down harmless bounded factors $(i\xi)^l$ and $\lambda(\xi)^m$.

\medskip
{\it (ii) (Proof of \eqref{finale}(2)).}
The second estimate on $s^p$ follows similarly, but substituting
for the estimate of term
$\int_{-\pi}^{\pi}
e^{i\xi x}\alpha(\xi) e^{\lambda(\xi) t}
 \langle \tilde \phi(\xi),\check g(\xi) \rangle_{L^2([0,1])}d\xi$
the slightly simpler estimate
$$
\begin{aligned}
\|\int_{-\pi}^{\pi}
e^{i\xi x}\alpha(\xi) e^{\lambda(\xi) t}
\langle \tilde \phi,\check g \rangle_{L^2([0,1])}d\xi\|_{L^p(\RM;dx)}
&\lesssim
\|\alpha(\xi) e^{- \eta \xi^2 t} |\langle \tilde \phi(\xi),\check g(\xi) \rangle_{L^2([0,1])}|
\|_{L^q(\xi,L^p([0,1]))}\\
&\lesssim
\|  e^{-\eta \xi^2 t} \|\check g(\xi)\|_{L^2([0,1];dx)} \|_{L^q([-\pi,\pi];d\xi)}\\
&\le
\|  e^{-\eta \xi^2 t}\|_{L^{rq}([-\pi,\pi];d\xi)}
\|g \|_{L^2(\RM;dx)}\\
&\lesssim
(1+t)^{-\frac{1}{2}(1/2-1/p)}\|g\|_{L^2(\RM)},
\end{aligned}
$$
where $1/r + 1/s=1$ and $qs=2$, so that $rq=\frac{2}{2-q}$ is
$\infty$ for $p=q=2$ and $2$ for $q=1$, $p=\infty$.
This verifies the result for $l=m=n=0$; estimates for general
$l,m,n\ge 0$ go similarly.

\medskip
{\it (ii) (Proof of \eqref{finalg}(1)).}
By (D2) and Pr\"uss' Theorem \cite{Pr},
we have
$$
|e^{L_\xi t}(1-\alpha(\xi))|_{H^{l+1}([0,1];dx)\to H^{l+1}([0,1];dx)},\;\;
|\alpha(\xi) e^{L_\xi t}\tilde \Pi|_{H^{l+1}([0,1];dx)\to H^{l+1}([0,1];dx)}
\lesssim e^{-\eta t},
\quad \eta>0,
$$
whence, by Sobelev embedding,
\[
|\partial_x^l e^{L_\xi t}(1-\alpha(\xi))|_{H^{l+1}([0,1];dx)\to L^p([0,1];dx)},\;\;
|\partial_x^j \alpha(\xi) e^{L_\xi t}\tilde \Pi|_{H^{l+1}([0,1];dx)\to L^p([0,1];dx)}
\lesssim e^{-\eta t}
\]
for $2\le p\le \infty$;
see \cite{JZ1} for details.\footnote{
Here, we are using the fact (see for example the Evans function
analysis of \cite{G}) that $H^{l+1}(\RM)$ and $L^2(\RM)$ spectra
coincide.}
The $W^{l,p}(\RM)$ norms of the first two terms of \eqref{tildeS}, by \eqref{hy} and
Parseval's identity, $\|\check g\|_{L^2(\xi,H^{l+1}([0,1]))}
\sim \|g\|_{H^{l+1}(\RM)}$,\footnote{
More precisely,
$$
\|g\|_{H^{l+1}(\RM)}^2=\|\check g\|_{L^2(\xi,L^2([0,1]))}^2+\|\check g\|_{L^2(\xi,\dot{H}^{l+1}_\xi([0,1]))}^2:=
\|\check g\|_{L^2(\xi,L^2([0,1]))}^2+\|(\partial_x +i\xi)^{l+1}\check g(\xi)\|_{L^2(\xi,L^2([0,1]))}^2.
$$
}
are thus bounded by $Ce^{-\eta t}\|g\|_{H^{l+1}(\RM)}$.
The $W^{l,p}(\RM)$ norm of the third term may be bounded similarly as in
the estimation of $s^p$ above, noting that the factor
$(\phi(\xi)-\phi(0))\sim |\xi|$ introduces an additional factor
of $(1+t)^{-1/2}$ decay.
This establishes the result for $m=n=0$; other cases go similarly,
noting that $\partial_t e^{L_\xi t} \tilde \Pi= L_\xi e^{L^\xi} \tilde \Pi$,
with $L_\xi$ a second-order operator, so that we may essentially trade
one $t$-derivative for two $x$-derivatives.

\medskip
{\it (ii) (Proof of \eqref{finalg}(2)).}
The second estimate on $\tilde S$ follows similarly, but substituting
for the estimate of term
$\int_{-\pi}^{\pi}
e^{i\xi x}\alpha(\xi) e^{\lambda(\xi) t}
(\phi(\xi,x)-\phi(0,x))
\langle \tilde \phi(\xi),\check g(\xi) \rangle_{L^2([0,1])}d\xi$
the estimate
$$
\begin{aligned}
\|\int_{-\pi}^{\pi}
e^{i\xi x}\alpha(\xi) e^{\lambda(\xi) t}
(\phi(\xi,x)-\phi(0,x))
\langle \tilde \phi(\xi),\check g(\xi) \rangle d\xi\|_{L^p(x)}
&\lesssim
\|\alpha(\xi) e^{- \eta \xi^2 t}|\xi| |\langle \tilde \phi(\xi),\check g(\xi) \rangle|
\|_{L^q(\xi,L^p(x))}\\
&\lesssim
\| |\xi| e^{-\eta \xi^2 t} \|\check g(\xi) \|_{L^2(x)} \|_{L^q(\xi)}\\
&\le
\| |\xi| e^{-\eta \xi^2 t}\|_{L^{rq}(\xi)}
\|g \|_{L^2(x)}\\
&\lesssim
(1+t)^{-\frac{1}{2}(1/2-1/p)-\frac{1}{2}}\|g\|_{L^2},
\end{aligned}
$$
where $1/r + 1/s=1$ and $qs=2$, so that $rq=\frac{2}{2-q}$ is
$\infty$ for $p=q=2$ and $2$ for $q=1$, $p=\infty$.
\end{proof}

\br\label{BG}
\textup{
In \cite{JZ1}, by contrast, $s^p(t)$ is estimated
through its integral kernel
$$
e(x,t;y)
:=(s^p(t) \delta_y)(x)=
\int_{-\pi}^{\pi} \alpha(\xi) e^{\lambda(\xi)t}e^{i\xi x}
\tilde \phi(\xi,y) d\xi ,
$$
yielding the slightly sharper estimate
$
\frac{\|s^p(t)g\|_{L^p(x)}}{\|g\|_{L^1(x)}}\le
\sup_y \|e(\cdot,t;y)\|_{L^p(x)}\le
\|\tilde \phi\|_{L^\infty(x,\xi)}(1+t)^{-1/2(1-1/p)}.
$
For our purposes, this makes no difference, and, as we shall
see in Section \ref{s:mod}, there can be an advantage
in maintaining the separation into distinct frequencies afforded
by the Bloch representation.
}
\er

\section{Linear behavior for modulational data}\label{s:mod}
Next, we consider behavior of \eqref{fullS}
when applied to modulational data
$g=h_0\bar u'$,
the linearized version of a nonlinear modulational perturbation
$\bar u(x+h_0(x)) -\bar u(x)\sim h_0(x)\bar u'(x)$.
The following estimates, obtained by frequency-domain rather than
spatial-domain (Green function) techniques
as in \cite{JZ1,JZ2,JZN},  
together with the associated
modified nonlinear iteration scheme of Section \ref{s:nonlin}, below,
represent the main new technical contributions of this paper.

\bpr\label{modprop}
Under assumptions (H1)--(H2) and (D1)--(D2),
for all $t>0$, $2\leq p\leq \infty$,
\be\label{Spmod}
\| \partial_x^l\partial_t^m s^p (t)  (h_0\bar u')\|_{L^p (\RM)
}
\lesssim
(1+t)^{-\frac{1}{2}(1-1/p)+\frac{1}{2}-\frac{l+m}{2}}
\|\partial_x h_0\|_{L^1 (\RM)
},
\ee
for
$l+m\ge 1$ or else $l=m=0$ and $p=\infty$,
and, for $0\leq l+2m \leq K+1$,
\be\label{tildeSmod}
\|\partial_x^l \partial_t^m \tilde S (t)(h_0\bar u')\|_{L^p
(\RM)
} \lesssim (1+t)^{-\frac{1}{2}(1-1/p)}
 \|\partial_x  h_0\|_{L^1 (\RM) \cap H^{l+2m+1} (\RM)
},
\ee
and when $t\le 1$
\be\label{spdiff}
\begin{array}{rcl}
\|\partial_x^l \partial_t^m (S^p (t) -\Id) (h_0\bar u')\|_{L^p (\RM) }
&\lesssim&\|\partial_x h_0\|_{L^1 (\RM)\cap H^{l+2m+1}(\RM)},\\
\|\partial_x^l \partial_t^m (s^p (t)(h_0\bar u')- h_0)\|_{L^p (\RM)}
&\lesssim&
\|\partial_x h_0\|_{L^1 (\RM)\cap L^{2}(\RM)}.
\end{array}
\ee
\epr

\begin{proof}
{\it (i) (Proof of \eqref{Spmod}), $l+m\ge 1$.}
We treat the case $l=1$, $m=0$; other cases go similarly.
First, re-express
\ba\label{Shf}
\partial_x (s^p(t) (\bar u' h_0))(x)&=
\int_{-\pi}^{\pi}
i\xi \alpha(\xi) e^{\lambda(\xi)t}e^{i\xi x} \langle \tilde \phi  \bar u', \check h_0 \rangle_{L^2([0,1])}(\xi) d\xi\\
&=
\sum_{j\in \ZM}\int_{-\pi}^{\pi}
i\xi \alpha(\xi) e^{\lambda(\xi)t}e^{i\xi x}
\langle \tilde \phi(\xi,y)  \bar u'(y),e^{i2\pi j y}
\rangle_{L^2([0,1])}\hat h_0(\xi+2j\pi) d\xi\\
&=
\sum_{j\in \ZM}\int_{-\pi}^{\pi}
i\xi \alpha(\xi) e^{\lambda(\xi)t}e^{i\xi x}
\widehat{\tilde\phi(\xi)\bar u'}_{j}^*
\hat h_0(\xi+2j\pi) d\xi\\
&=
\sum_{j\in \ZM}\int_{-\pi}^{\pi}
\frac{ \xi}{\xi+2\pi j} \alpha(\xi) e^{\lambda(\xi)t}e^{i\xi x}
\widehat{\tilde\phi(\xi)\bar u'}_{j}^*
\widehat{ \partial_x h_0}(\xi+2j\pi) d\xi,\\
\ea
where
$\widehat{\tilde\phi(\xi)\bar u'}_{j}$
 denotes the $j$th Fourier coefficient in the
Fourier expansion of periodic function
$\tilde\phi(\xi)  \bar u'$. By \eqref{hy}, we thus get
$
\|\partial_x (s^p (t)(\bar u' h_0))\|_{L^p(\RM)}
\le C(1+t)^{-\frac{1}{2}(1-1/p)} \|\partial_x h_0\|_{L^1(\RM)}
\sum_j \Big| \frac{\widehat{\tilde\phi(\xi)\bar u'}_{j}}{1+|j|} \Big|,
$
yielding the result
together with the Cauchy--Schwarz estimate
$$
\sum_j \Big| \frac{\widehat{\tilde\phi(\xi)\bar u'}_{j}}{1+|j|} \Big|
\le
\sqrt{
\sum_j (1+|j|)^{-2} \sum_j |\widehat{\tilde\phi(\xi)\bar u'}_{j}|^2}
\lesssim \|\tilde \phi(\xi)\bar u'\|_{L^2(\RM)}.
$$

\medskip

{\it (ii) (Proof of \eqref{Spmod}), $l=m=0$.}
This follows by an explicit error function decomposition as in \cite{JZ2}, p. 18,
putting a $\xi$ factor on $h_0\bar u'$ and
a $1/\xi$ factor on $e^{\lambda(\xi)t}$ and proceeding as in the cited
estimate.
More precisely, for the low-frequency part of $\hat h_0$, express
the principal part
$ \int_{-\pi}^{\pi}e^{i\xi x} \frac{e^{(ai\xi-d\xi^2)t}}{\xi} d\xi $ as
$$
\int_{-\infty}^{+\infty}e^{i\xi x} \frac{e^{(ai\xi-d\xi^2)t}}{i\xi} d\xi
={\rm errfn} ((x-at)^2/\sqrt{t})
$$
minus
$ \sqrt{t}\int_{|\xi|\ge \pi}
e^{i\xi x} \frac{e^{(ai\xi-d\xi^2)t}}{\xi t^{1/2}} d\xi $,
estimating the $L^\infty(\RM)$ norm of convolutions with the
former by the Triangle Inequality,
and with the latter by Hausdorff--Young's inequality;
for details, see \cite{JZ2}.\footnote{
This estimate is not needed for the nonlinear iteration, so we
do not emphasize it here.  We note that the resulting bounds
\eqref{andpsi} may be recovered alternatively by the modulation decomposition
of \cite{JNRZ1} without case $l=0$.
}

\medskip

{\it (iii) (Proof of \eqref{tildeSmod}).}
Likewise, this follows by re-expressing
\ba\label{tildeShf}
\tilde S(t) (h_0 \bar u')(x)&=
\int_{-\pi}^{\pi}
e^{i\xi x} (1-\alpha(\xi))(e^{L_\xi t}
\check h_0(\xi)\bar u')(x) d\xi
+ \int_{-\pi}^{\pi} e^{i\xi x} \alpha(\xi)
(e^{L_\xi t} \tilde \Pi(\xi) \check h_0(\xi) \bar u')(x) d\xi\\
&\quad
+ \int_{-\pi}^{\pi}
e^{i\xi x}\alpha(\xi) e^{\lambda(\xi) t} (\phi(\xi,x)-\phi(0,x))
\langle \tilde \phi(\xi), \check h_0(\xi)\bar u'\rangle_{L^2([0,1])} d\xi\\
&=
\sum_{j\in \ZM} \int_{-\pi}^{\pi}
e^{i\xi x}(1-\alpha(\xi))
\frac{( e^{L_\xi t}(\bar u' e^{2i\pi j\cdot}))(x)}{i(\xi+2\pi j)}
\widehat {\partial_x h_0}(\xi+2j\pi) d\xi \\
&\quad+ \sum_{j\in \ZM} \int_{-\pi}^{\pi}
e^{i\xi x}\alpha(\xi)
\frac{(e^{L_\xi t}\tilde \Pi(\xi)(\bar u' e^{2ij\pi \cdot}))(x)}{i(\xi+2\pi j)}
\widehat {\partial_x h_0}(\xi+2j\pi) d\xi \\
&\quad+ \sum_{j\in \ZM} \int_{-\pi}^{\pi}
e^{i\xi x}\alpha(\xi) e^{\lambda(\xi)t}
\frac{ (\phi(\xi,x)-\phi(0,x))\widehat{\tilde \phi \bar u'}_j(\xi)^*}{i(\xi+2\pi j)}
\widehat {\partial_x h_0}(\xi+2j\pi) d\xi \\
\ea
then estimating as before, where we are using
$(1-\alpha(\xi))\lesssim \xi$,
$\tilde \Pi(\xi) \bar u'=O(\xi)$ ,
$\phi(\xi)-\phi(0)=O(\xi)$,
 respectively,
to bound the key terms
$\frac{ (1-\alpha(\xi))}{i(\xi+2\pi j)}$ ,
$\frac{ \tilde \Pi(\xi)(\bar u' e^{2ij\pi \cdot})}{i(\xi+2\pi j)}$ ,
and
$\frac{\phi(\xi)-\phi(0)}{i(\xi+2\pi j)}$
for $j=0$, and are using the Cauchy--Schwarz inequality
$$
\sum_j \left|\frac{\widehat {\partial_x h_0}(\xi+2j\pi)}{\xi+2j\pi j}\right|
\le C\sqrt{ \sum_j \frac{1}{(1+|j|)^2} \sum_{j'}
|\widehat{\partial_x h_0}(\xi+2j'\pi)|^2}
\le C\|\partial_x h_0\|_{L^2(\RM)}
$$
to bound the sum over $j\neq0$. This establishes the case $l=m=0$; other cases go similarly.

\medskip

{\it (iv) (Proof of \eqref{spdiff}).}
Expanding
$
S^p(t)-\Id=(S^p(t)-S^p(0)) -\tilde S(0)
=t\partial_t S^p(s(t)) -\tilde S(0)
$
for some $0<s(t)<t$,
we obtain the first inequality by combining \eqref{Spmod}
and \eqref{tildeSmod}.
Likewise, the second ineqality follows by expanding
$
s^p(t)(h_0 \bar u')-h_0=(s^p(t)-s^p(0))(h_0 \bar u')
+ (s^p(0)(h_0\bar u') - h_0)
$
and applying \eqref{Spmod} together with
$$
\begin{array}{l}
\|s^p(0)(h_0\bar u') - h_0\|_{L^p(\RM)}\\
=
\Big\|\int_{-\pi}^{\pi}e^{i\xi x}(1-\alpha(\xi))
\sum_j \hat h_0(\xi+2j\pi) \widehat {\tilde \phi(\xi)\bar u'}_j^*d\xi
\ +\int_{-\pi}^{\pi}e^{i\xi x}
\sum_{j\ne 0} \hat h_0(\xi+2j\pi)
\Big(\widehat {\tilde \phi(\xi) \bar u'}_j^*-1\Big)d\xi\Big\|_{L^p(\RM)}\\
\le
C \sup_\xi\sum_j
\frac{|\xi||\widehat {\tilde \phi(\xi) \bar u'}_j|}{|\xi +2j\pi|} \|\partial_x h_0\|_{L^1(\RM)}
+ C \sup_\xi\sum_{j\ne 0} \frac{|\widehat{\partial_x h_0}(\xi+2j\pi)|} {|\xi +2j\pi|}
\le C\|\partial_x h_0\|_{L^1(\RM)\cap L^2(\RM)}
\end{array}
$$
This establishes the case $l=m=0$; other cases go similarly.
\end{proof}

\begin{remark}\label{remarks}
\textup{
1. If we split $h_0$ into high-frequency and low-frequency parts,
then the contribution of the high-frequency part decays faster by
factor $(1+t)^{-1/2}$ in
all estimates \eqref{Spmod}--\eqref{spdiff}. 
}

\textup{
2. In the estimate of $s^p$, it is easy to see that the same bounds
hold if $\bar u'$ is replaced by any periodic $f\in H^1_{\rm per}([0,1])$.
However, in the estimate of $\tilde S$, replacing $\bar u'$ by
periodic $f\in H^1_{\rm per}([0,1])$ introduces an exponentially decaying error
in $f-\bar u'$,  reflecting the dynamics of the
problem on a periodic domain.
To extend our results to the sum of a localized
perturbaton and a perturbation asymptotic as $x\to \pm \infty$
to any two fixed periodic waves, not necessarily modulations,
requires only the estimation of this latter error term as
a time-exponentially function from $L^\infty(\RM)\to L^\infty(\RM)$,
a semigroup/Fourier multiplier problem of concrete technical nature.
This would be an interesting direction for further investigation.
}

\textup{
3. The estimate \eqref{spdiff} is the key to handling the
``initial layer problem'' in our later nonlinear iteration,
allowing us to essentially prescribe initial data for $\psi$
as is convenient for the analysis; see Remark \ref{remarks2} below for more information.
}
\end{remark}

\section{Nonlinear perturbation equations}\label{s:pert}
Essentially following \cite{JZ1}, for $\tilde{u}(x,t)$ satisfying
$\ks \tilde{u}_t=\ks^2\tilde{u}_{xx}+f(\tilde{u})+\ks c\tilde{u}_x$
and $\psi(x,t)$ to be determined, set
\begin{equation}
u(x,t)=\tilde{u}(x-\psi(x,t),t)\quad\textrm{and}\quad
v(x,t)=u(x,t)-\bar{u}(x)\label{pertvar}.
\end{equation}

\begin{lemma}[\cite{JZ1}]\label{lem:canest}
The nonlinear residual $v$ defined in \eqref{pertvar} satisfies
\be\label{veq}
\ks\left(\partial_t-L\right)(v+\psi\bar u_x)=\ks\mathcal{N},
\qquad
\ks\mathcal{N}=
\cQ+ \cR_x +(\ks\partial_t+\ks^2\partial_x^2)\cS+\cT,
\ee
where
\be\label{eqn:Q}
\cQ:=f(v+\bar{u})-f(\bar{u})-df(\bar{u})v
\ee
\be\label{eqn:R}
\cR:= -\ks v\psi_t -\ks^2 v\psi_{xx}+ \ks^2(\bar u_x +v_x)\frac{\psi_x^2}{1-\psi_x},
\ee
\be\label{eqn:S}
\cS:= v\psi_x
\ee
and
\be\label{eqn:T}
\cT:=-\left(f(v+\bar{u})-f(\bar{u})\right)\psi_x.
\ee
\end{lemma}

\begin{proof}
(See also \cite{JZ1}.\footnote{
Note that we have here followed a different convention
than in \cite{JZ1}, reversing the sign of $\psi$ in \eqref{pertvar}
in agreement with formal asymptotics of \cite{Se,DSSS}.
The change $\psi\to -\psi$ recovers the formulae of \cite{JZ1}.})
From definition \eqref{pertvar} and the fact that $\tilde u$ satisfies
\eqref{rd}, we obtain
\be
\ks(1-\psi_x)u_t+\ks(-c+\psi_t)u_x=\ks^2\left(\frac{1}{1-\psi_x}u_x\right)_x\ +\ (1-\psi_x)f(u)
\ee
hence, subtracting the profile equation $-k_* c\bar u_x=
k_*^2 \bar u_{xx} + f(\bar u)$ for $\bar u$,
\be
\ks v_t-\ks Lv+\ks\psi_t v_x-\ks\psi_x v_t+\ks\psi_t \bar u_x=\cQ-\psi_x f(u)
+\ks^2\left(\frac{\psi_x}{1-\psi_x}u_x\right)_x
\ee
or
\be\label{int}
\ks(\partial_t-L)v+\ks(\psi\bar u_x)_t=\cQ+\cR_x+(\ks\partial_t+\ks^2\partial^2_x)\cS-\psi_x f(u)+\ks^2\left(\psi_x\bar u_x\right)_x,
\ee
and we finish the proof using $L\bar u_x=0$ and the profile equation for $\bar u$ to obtain
\be
\ks L(\psi \bar u_x) = \ks c \psi_x\bar u_x+\ks^2\psi_x\bar u_{xx}+\ks^2\left(\psi_x\bar u_x\right)_x=-\psi_x f(\bar u)+\ks^2\left(\psi_x\bar u_x\right)_x\ .
\ee
\end{proof}

\section{Nonlinear damping estimate}\label{s:damping}

\begin{proposition}[\cite{JZ1}]\label{damping}
Assuming $(H1)-(H2)$, let $v(\cdot,0)\in H^K(\RM)$ (for $v$ as in \eqref{pertvar}) and suppose that for some $T>0$,
the $H^K(\RM)$ norm of $v(t)$ and $\psi_t(t)$ and the $H^{K+1}(\RM)$ norm of $\psi_x(t)$ remain bounded by a sufficiently small constant for all $0\leq t\leq T$.  Then there
 are positive constant $\theta$ and $C$, independent of $T$,
 such that, for all $0\leq t\leq T$,
\be\label{Ebds}
\|v(t)\|^2_{H^K(\RM)}
\leq C\,e^{-\theta t}
\|v(0)\|_{H^K(\RM)}^2+
C\int_0^t e^{-\theta(t-s)}
\left(\|v(s)\|^2_{L^2(\RM)}+
\|(\psi_t, \psi_x)(s)\|_{H^K(\RM)}^2\right)ds.
\ee
\end{proposition}

\begin{proof}
Take for writing simplicity $k_*=1$.
Rewriting \eqref{int} as
\ba\label{vperturteq2}
(1-\psi_x) v_t -v_{xx}-cv_x&=
df(\bar u)v+ \cQ
- (\bar u_x+v_x) \psi_t
\\ &\quad
+ ((\bar u_x+v_x)  \psi_x)_x
+ \Big((\bar u_x+v_x) \frac{\psi_x^2}{1-\psi_x}\Big)_x-f(\bar u+v)\psi_x,
\ea
taking the $L^2$ inner product against
$\sum_{j=0}^K \frac{(-1)^{j}\partial_x^{2j}v}{1-\psi_x}$,
integrating by parts, and rearranging, we obtain
\[
\frac{d}{dt}
\|v\|_{H^K(\RM)}^2(t)
 \leq -\tilde{\theta} \|\partial_x^{K+1} v(t)\|_{L^2(\RM)}^2 +
C\left( \|v(t)\|_{H^K(\RM)}^2
+\|(\psi_t, \psi_x)(t)\|_{H^K(\RM)}^2 \right),
\]
for some $\tilde\theta>0$, so long as
$\|(v,\psi_t, \psi_x,\psi_{xx})(t)\|_{H^{K+1}(\RM)}$ remains sufficiently small.
Sobolev interpolation
$
\|g\|_{H^K(\RM)}^2 \leq  \tilde{C}^{-1}\|\partial_x^{K+1} g\|_{L^2(\RM)}^2 + \tilde{C} \| g\|_{L^2(\RM)}^2
$
gives, then, for $\tilde{C}>0$ sufficiently large,
\[
\frac{d}{dt}\|v\|_{H^K(\RM)}^2(t) \leq -\theta \|v(t)\|_{H^K(\RM)}^2 +
C\left( \|v(t)\|_{L^2(\RM)}^2+\|(\psi_t, \psi_x)(t)\|_{H^K(\RM)}^2 \right),
\]
from which \eqref{Ebds} follows by Gronwall's inequality.
See \cite{JZ1} for further details.
\end{proof}

\section{Nonlinear iteration scheme}\label{s:nonlin}

The key idea is, similarly as in the localized case treated in
\cite{JZ1}, starting with
$$
(\partial_t-L)(v+\psi \bar u')=\mathcal{N},
\qquad
v|_{t=0}=d_0,\ \psi|_{t=0}=h_0,
$$
where $ d_0:=\tilde u_0(\cdot-h_0(\cdot))-\bar u \in L^1(\RM)\cap H^K(\RM)$,
$\partial_x h_0\in L^1(\RM)\cap H^K(\RM)$, to
choose $\psi$ to cancel $s^p$ contributions, as
\ba\label{psidef}
\psi(t)&=
s^p(t) (h_0\bar u'+d_0)+
\int_0^t s^p(t-s)\N(s)ds\\
&-(1-\chi(t))\left(s^p(t)(d_0+h_0\bar u')-h_0+\int_0^t s^p(t-s)\N(s)ds\right),
\ea
where $\chi(t)$ is a smooth cutoff that is zero for $t\le 1/2$ and one for $t\ge 1$, leaving the system
\ba\label{closed}
v(t)&=\tilde S(t) (d_0+h_0\bar u') + \int_0^t \tilde S(t-s) \N(s)ds\\
&\quad
+(1-\chi(t))\left(S^p(t)d_0+(S^p(t)-\Id)h_0\bar u'
+\int_0^t S^p(t-s)\N(s)ds\right) .
\ea
We may extract from \eqref{psidef}-\eqref{closed} a closed system in
$(v,\psi_x,\psi_t)$ (and some of their derivatives),
 and then recover $\psi$ through the slaved equation \eqref{psidef}.

\br\label{trick}
\textup{
At first sight, we have accomplished nothing by
introducing a $\psi$-dependent change of variable and choosing $\psi(0)=h_0$,
since we have replaced the nonlocalized perturbation
$\tilde u_0(x)-\bar u(x)$ used in the previous $h_0=0$ setup
of \cite{JZ1},
by a different nonlocalized perturbation $d_0 + h_0 \bar u'$.
However, what we really did was replace the asymptotic
states $\bar u(x+c_\pm)-\bar u(x)$ by their linear approximates
$c_\pm \bar u'(x)$, which removes the key difficulty of higher order
remainders in the Taylor expansion of nonlinear modulations.
}
\er

\br\label{remarks2}
Notice that modulational data $\bar u' h_0$
enters in \eqref{closed} only through operators $\tilde S(t)$ and
$(1-\chi(t))(S^p(t)-\Id)$ for which we have Gaussian decay in $L^p(\R)$
with respect to
\footnote{For simplicity we set $l=m=0$ in this discussion.}
 $\|\partial_x h_0\|_{L^1(\R)\cap H^1(\RM)}$,
hence the error incurred by defining $\psi$ by \eqref{psidef}
instead of the value $ \tilde \psi(t)=s^p(t) (h_0\bar u'+d_0)+ \int_0^t s^p(t-s)\N(s)ds$
exactly canceling $s^p$ terms is harmless to our analysis.
The choice of \eqref{psidef} reflects our need to accomodate the incompatibility
between the initial value $\psi|_{t=0}=h_0$ prescribed by the
spatially-asymptotic behavior of the initial perturbation
and the function $\tilde \psi$ encoding time-asymptotic behavior of the
perturbed solution; that is, it is a device to avoid having to resolve
an initial layer near $t=0$.\footnote{In the case $h_0\equiv 0$,
this essentially reduces to the simpler device used in \cite{JZ1} to treat
the localized case, of substituting $\chi(t)s^p(t)$ for $s^p(t)$.
However, the latter is clearly too crude to treat the present case.}
Whether this initial layer is an artifact of our analysis
or reflects some aspect of short-time behavior is unclear;
as the estimates show, this is below our level of resolution.
\er

\section{Nonlinear stability}\label{s:proof}
With these preparations, the proof of stability now goes essentially as in the localized conservative case treated in \cite{JZ2}, using the new linear modulation bounds to estimate the new linear term coming from data
$h_0\bar u'$ in \eqref{psidef} and \eqref{closed}.
As noted in \cite{JZ1}, from differential equation \eqref{veq} together with
integral equation \eqref{closed}
\eqref{psidef}, we readily obtain
short-time existence and continuity with respect to $t$ of solution
$(v,\psi_t,\psi_x)\in H^{K}(\RM)$ by a standard
contraction-mapping argument treating the
linear $df(\bar u)v$ term of the lefthand side along with $\cQ,\cR,\cS,\cT,\psi \bar u'$
terms of the righthand side as sources in the heat equation.
Associated with this solution define so long as it is finite,
\ba\label{szeta}
\zeta(t)&:=\sup_{0\le s\le t}
 \|(v, \psi_t,\psi_x)(s)\|_{H^K(\RM)}(1+s)^{1/4} .
\ea

\bl\label{sclaim}
For all $t\ge 0$ for which $\zeta(t)$ is
finite and sufficiently small,
some $C>0$,
and $E_0:=\|(d_0,\partial_x h_0)\|_{L^1(\RM)\cap H^K(\RM)}$ sufficiently small,
\be\label{eq:sclaim}
\zeta(t)\le C(E_0+\zeta(t)^2).
\ee
\el

\begin{proof}\footnote{
Compare to the argument of Lemma 4.2, [JZ2], regarding localized perturbations in the conservative case.} 
By \eqref{eqn:Q}--\eqref{eqn:T}
and corresponding bounds on the derivatives together
with definition \eqref{szeta}, and using \eqref{vperturteq2} to bound $v_t$,
\be\label{sNbds}
\| \N(t)\|_{L^1(\RM)\cap H^1(\RM)}
\lesssim \|(v,\psi_t,\psi_x)(t)\|_{H^3(\RM)}^2
\le C\zeta(t)^2 (1+t)^{-\frac{1}{2}},\\
\ee
so long as $\zeta(t)$
remains small.
Applying the bounds
\eqref{finale}(1)--\eqref{finalg}(1)
and \eqref{Spmod}--\eqref{spdiff}
of
Propositions \ref{greenbds} and \ref{modprop}
to system \eqref{psidef}- \eqref{closed}, we obtain for any $2\le p<\infty$
\ba\label{sest}
\|v(t)\|_{L^p(\RM)}& \le
C(1+t)^{-\frac{1}{2}(1-1/p)}E_0
+
C\zeta(t)^2\int_0^{t} (1+t-s)^{-\frac{1}{2}(1-1/p)-\frac{1}{2}}
(1+s)^{-\frac{1}{2}}ds\\
&
\le
 C_p (E_0+\zeta(t)^2) (1+t)^{-\frac{1}{2}(1-1/p)}
\ea
and
\ba\label{sestad}
\|(\psi_t,\psi_x)(t)\|_{W^{K+1,p}(\RM)}& \le
C(1+t)^{-\frac{1}{2}}E_0 +
C\zeta(t)^2\int_0^{t} (1+t-s)^{-\frac{1}{2}(1-1/p)-1/2}
(1+s)^{-\frac{1}{2}}ds \\
&\le
 C_p
(E_0+\zeta(t)^2) (1+t)^{-\frac{1}{2}(1-1/p)}.
\ea
Estimate \eqref{sestad} yields in particular that
$\|(\psi_t,\psi_x)(t)\|_{H^{K+1}(\RM)}$ is small,
verifying the hypotheses of Proposition
\ref{damping}.
From \eqref{Ebds} and \eqref{sest}--\eqref{sestad},
we thus obtain
$$\|v(t)\|_{H^K(\RM)} \le
 C(E_0+\zeta(t)^2) (1+t)^{-\frac{1}{4}}.$$
Combining this with \eqref{sestad}, $p=2$, rearranging, and recalling
definition \eqref{szeta}, we obtain the result.
\end{proof}

\begin{proof}[Proof of Theorem \ref{main}]
By short-time $H^K$ existence theory,
$\|(v,\psi_t,\psi_x)(t)\|_{H^{K}(\RM)}$
is continuous so long as it
remains small, hence $\zeta$ remains
continuous so long as it remains small.
By \eqref{sclaim}, therefore,
it follows by continuous induction that
$\zeta(t) \le 2C E_0$ for $t \ge0$, if $E_0 < 1/ 4C$,
yielding by (\ref{szeta}) the result (\ref{mainest}) for $p=2$.
Applying \eqref{sest}--\eqref{sestad}, we obtain
(\ref{mainest}) for $2\le p\le p_*$ for any $p_*<\infty$,
with uniform constant $C$.
Taking $p_*>4$ and estimating
\be\label{QRST}
\|(\cQ,\cR,\cS,\cT)(t)\|_{L^2(\RM)} \lesssim \|(v,\psi_t,\psi_x,\psi_{xx})(t)\|_{L^4(\RM)}^2\le CE_0(1+t)^{-\frac{3}{4}}
\ee
in place of the weaker \eqref{sNbds}, then applying \eqref{finale}(ii) in place of \eqref{finale}(i),
we obtain\footnote{We bound the $\d_t\cS$ contribution according to $$\int_0^t s^p(t-s)\d_t\cS(s)ds=-\int_0^t \d_t[s^p](t-s)\cS(s)ds+s^p(0)\cS(t)-s^p(t)\cS(0).$$}
\ba\label{sestad2}
\|(\psi_t,\psi_x)(t)\|_{W^{K+1,p}(\RM)}& \le
C(1+t)^{-\frac{1}{2}}E_0 +
C\zeta(t)^2\int_0^{t} (1+t-s)^{-\frac{1}{2}(1/2-1/p)-1/2}
(1+s)^{-\frac{3}{4}}ds \\
&\le C (E_0+\zeta(t)^2) (1+t)^{-\frac{1}{2}(1-1/p)},
\ea
for $2\le p\le \infty$.
Likewise, using \eqref{QRST} together with bound
$$
\| (\cQ,\cT)(t)\|_{H^1(\RM)} +\| \cR(t)\|_{H^2(\RM)}
+\| \cS(t)\|_{H^3(\RM)}
\lesssim \zeta(t)^2 (1+t)^{-\frac{1}{2}}
$$
obtained in the course of proving \eqref{sNbds},
we may use \eqref{finalg}(ii) rather than \eqref{finalg}(i) to get
\ba\label{sest2}
\|v(t)\|_{L^p(\RM)}& \le
C(1+t)^{-\frac{1}{2}(1-1/p)}E_0
+C\zeta(t)^2\int_0^t e^{-\eta (t-s)}(1+s)^{-\frac{1}{2}}ds\\
&+C\zeta(t)^2\int_0^{t} (1+t-s)^{-\frac{1}{2}(1/2-1/p)-\frac{1}{2}}(1+s)^{-\frac{3}{4}}ds\\
&\le C(E_0+\zeta(t)^2) (1+t)^{-\frac{1}{2}(1-1/p)}
\ea
and achieve the proof of \eqref{mainest} for $2\le p\le \infty$.

Estimate \eqref{andpsi} then follows through
\eqref{psidef}
using \eqref{finale}(i),
by
\ba\label{sesta}
\|\psi (t) \|_{L^\infty(\RM)}& \le
C E_0
+
C\zeta(t)^2\int_0^{t} (1+t-s)^{-\frac{1}{2}}
(1+s)^{-\frac{1}{2}}ds
 \le C(E_0+\zeta(t)^2),
\ea
yielding nonlinear stability by the fact that
\be
\tilde u(x-\psi(x,t),t)-\bar u(x-\psi(x,t))= v(x,t)+ \bar u(x)-\bar u(x-\psi(x,t)),
\ee
so that $|\tilde u (t) -\bar u|$ is controlled
by the sum of $|v(t)|$ and
$|\bar u-\bar u(\cdot-\psi(\cdot,t))|\lesssim |\psi(t)|\, \sup|\bar u'|$.
\end{proof}

\medskip

{\bf Acknowledgement.} Thanks to B\"jorn Sandstede for
pointing out the results of \cite{SSSU}.

\end{document}